\documentclass[a4paper,12pt]{amsart}

\usepackage{amsmath,amsfonts,latexsym,amssymb,amscd}
\usepackage{amsthm,latexsym,epic,eepic,amsbsy,amstext,amsxtra,latexsym, xypic}
\usepackage{CJK,pshan}
\usepackage{indentfirst}
\usepackage{epsfig}
\usepackage{texdraw}
\usepackage{color}
\usepackage{tikz}
\usetikzlibrary{matrix}
\theoremstyle{plain}
\newtheorem{thm}{Theorem}[section]
\newtheorem{theorem}[thm]{Theorem}

\newtheorem{lemma}[thm]{Lemma}
\newtheorem{corollary}[thm]{Corollary}
\newtheorem{proposition}[thm]{Proposition}
\theoremstyle{definition}
\newtheorem{remark}[thm]{Remark}

\newtheorem{notation}[thm]{Notation}
\newtheorem{notation-definition}[thm]{Notation-Definition}
\newtheorem{notation-remark}[thm]{Notation-Remark}
\newtheorem{definition}[thm]{Definition}

\numberwithin{equation}{section}

\newcommand{\Sym}{{\rm Sym}}

\title{Normalization of congruence of bitangents to a hypersurface in $\mathbb P^3$}
\author{Hosung Kim and Yongnam Lee}

\address {Center for Complex Geometry\\
Institute for Basic Science (IBS)\\
55 Expo-ro, Yuseong-gu\\ 
Daejeon, 34126 Korea}
\email{hosung@ibs.re.kr}

\address {Center for Complex Geometry\\
Institute for Basic Science (IBS)\\
55 Expo-ro, Yuseong-gu\\ 
Daejeon, 34126 Korea,  and 
\newline\hspace*{3mm} Department of Mathematical Sciences\\
KAIST\\
291 Daehak-ro, Yuseong-gu\\ 
Daejeon, 34141 Korea}
\email{ynlee@kaist.ac.kr}

\thanks{
Lee was partly supported by Samsung Science and Technology Foundation under Project Number SSTF-BA1701-04. Kim was supported by the Institute for Basic Science (IBSR032-D1).
Lee would also like to  acknowledge the support and hospitality of KIAS when he visited as Affiliate Professor.}
\date{\today}

\begin{document}

\begin{abstract}A congruence is a surface in the Grassmannian ${\rm Gr}(2, 4)$.  In this paper, we consider the normalization of congruence of bitangents to a hypersurface in $\mathbb P^3$. We call it the Fano congruence of bitangents. We give a criterion for smoothness of the Fano congruence of bitangents and describe explicitly their degenerations in a general Lefschetz pencil in the space of hypersurfaces in $\mathbb P^3$.\end{abstract}

\maketitle

\section{Introduction}
Throughout this paper we will work over the field of complex numbers. 

Let ${\rm Gr}(2,4)$ be  the  Grassmannian of lines in $\mathbb P^3$. A surface in ${\rm Gr}(2,4)$ is called a congruence (cf. \cite{ABT01}).
Let $Y$ be a  hypersurface of degree $d\geq 4$ in $\mathbb P^3$. By the congruence of bitangents to $Y$ we mean the space $B(Y)$  of  lines in $\mathbb P^3$ which is tangent to $Y$ at two points.  Indeed $B(Y)$ is a smooth irreducible surface if $Y$ is a smooth quartic and contains no line (cf \cite{T80}). 
For $d\geq 5$, if $Y$ has only isolated singularities and its dual variety $Y^*\subset (\mathbb P^3)^*$ is an irreducible hypersurface, then $B(Y)$ is a congruence (see Proposition~\ref{p.congruence}). It is well known that $B(Y)$ is not smooth unless $Y$ is a quartic surface even when $Y$ is general (see Lemma 4.3 in \cite{ABT01}).

In this paper we are going to study  the Fano congruence of bitangents to a hypersurface of degree $d\geq 5$ in $\mathbb P^3$  defined as follows:

\begin{definition}\label{d.Fano congruence}
	Let $Y$ be a  hypersurface in $\mathbb P^3$ of degree $d\geq 5$. The Fano congruence  of bitangents to $Y$ is the space of  $P=(\ell,p_1+p_2, q_1+\cdots+q_{d-4})$ with $\ell\in {\rm Gr}(2,4)$ and $p_i,q_j\in\ell$ such that $Y\cap\ell$ contains $2p_1+2p_2+q_1+\cdots+q_{d-4}$ as a subscheme. We denote the Fano congruence of bitangents to $Y$ by $S(Y)$ in this paper.
\end{definition} Let $\phi_Y:S(Y)\rightarrow B(Y)$ be the natural morphism. Then  $\phi_Y$ is a finite morphism if $Y$ contains no line. 
The space $B(Y)$ had been studied in detail in \cite{ABT01} and \cite{KNT}. But in this paper, we focus on  $S(Y)$ instead of $B(Y)$ because $S(Y)$ is a smooth projective surface for general $Y$ (see Proposition~\ref{p.general smooth}) and their degenerations in a general Lefschetz pencil in the space of hypersurfaces $Y$ can be described (see Theorem~\ref{t.main}). In particular, special fiber in a general Lefschetz pencil admits at worst isolated singularities or has  double points of rank two along a smooth irreducible curve.
For arbitrary $Y$,  $S(Y)$ is smooth at any point $P$ such that $\ell\not\subset Y$, $Y$ is smooth at $p_1$ and $p_2$, and $\{p_1,p_2\}\cap \{q_1,\ldots,q_{d-4}\}=\emptyset$ (see Theorem \ref{t.smooth}). 

The subject of this paper is to prove the following.
\begin{theorem}\label{t.main}Let $\pi_L:\mathcal Y=\{Y_b\}_{b\in L}\rightarrow L$ be a general Lefschetz pencil of hypersurfaces of degree $d\geq 5$ in $\mathbb P^3$ and  
	let $b_1,\ldots,b_N\in L$ be the singular values of $\pi_L$. Then each $b\in L$, $S(Y_b)$ is a projective surface with singularities as follows.  
	\begin{itemize}	 
		\item[(i)] If $S(Y_b)$ is singular at some point  $P=(\ell,p_1+p_2,q_1+\cdots+q_{d-4})\in S(Y_b)$ such that    $Y_b$ is smooth at $p_1$ and $p_2$, then $P$ is an isolated singular point of $S(Y)$, and the followings are  satisfied:
		\begin{itemize}
			\item [(a)] $p_1\neq p_2$,  $\{p_1,p_2\}\cap \{q_1,\ldots,q_{d-4}\}\neq\emptyset$, and $q_1,\ldots,q_{d-4}$ are distinct.  
			\item [(b)] if $p_1\in\{q_1,\ldots,q_{d-4}\}$ and $p_2\not\in \{q_1,\ldots,q_{d-4}\}$, then  $\hat T_{p_1}Y_b=\hat T_{p_2}Y_b$ and $\hat T_{p_1}Y_b\cap Y_b$ is a plane curve with a cuspidal singularity at $p_1$. Here $\hat T_{p_i}Y_b\subset \mathbb P^3$ denotes the embedded tangent space of $Y_b$ at $p_i$. 
		\end{itemize}
		\item[(ii)] For each $1\leq J\leq N$, let $p_J$ be the node of $Y_{b_J}$ and let  $\Gamma_J$ be the space of  $P=(\ell,p_1+p_2,q_1+\cdots+q_{d-4})$ with $p_J\in\{p_1,p_2\}$ and $P\in S(Y_{b_J})$. Then $\Gamma_J$ is a smooth projective curve and  $S(Y_{b_J})$ has  double points of rank two along  $\Gamma_J$. 
		\item[(iii)]for all $b\in L\setminus \{b_1,\ldots,b_N\}$, $\phi_{Y_b}:S(Y_b)\rightarrow B(Y_b)$ is the normalization. 
	\end{itemize}
\end{theorem}

In order to obtain the above theorem, we consider $S(Y)$ as a subscheme of the fiber product of  two projective vector bundles
$$\mathcal F=\mathbb P(\Sym^2\mathcal I^{\vee})\times_{{\rm Gr}(2, 4)}\mathbb P(\Sym^{d-4}\mathcal I^{\vee})$$ on ${\rm Gr}(2,4)$ where $\mathcal I^{\vee}$ is the dual of the tautological vector bundle $\mathcal I$ on ${\rm Gr}(2,4)$. We also combine local computational analysis with global methods from intersection theory.
\bigskip

{\bf  Notation and Convention} In this paper we will use the following notational conventions.
\begin{itemize}\item[(i)] For a vector space $V$, the projective space $\mathbb P(V)$ is the space of 1-dimensional subspaces of $V$. 
\item[(ii)] Let  $Y$ be a  degree $d$ hypersurface of $\mathbb P^n$.  For a finite sequence of nonnegative   integers $a_1\geq \cdots\geq a_r$ with $\sum a_i=d$, we denote by  $S_{a_1,\ldots,a_r}(Y)\subset {\rm Gr}(2,n+1)$  the space of lines $\ell$ in $\mathbb P^n$ such that  $\ell\cap Y$ contains $a_1p_1+\cdots+a_rp_r$ as a subscheme for some $p_i\in \ell$.

For any positive integers $n_1,\ldots,n_r$, we will also use the notation    $S_{a_1^{n_1},a_2^{n_2},\ldots,a_r^{n_r}}(Y)$ to denote $$S_{a_1,\ldots,a_1,a_2,\ldots,a_2,\ldots,a_r,\ldots,a_r}(Y)$$ here each $a_i$ occurs $n_i$ times.  For example,  $B(Y)=S_{2^2,1^{d-4}}(Y)$. 
\item[(iii)]Let $\mathbb C[t_0,t_1]_d$ be the space of homogeneous polynomials  of degree $d$ in variables $t_0$ and $t_1$. For finite subsets $A_1,\ldots,A_r\subset  \mathbb C[t_0,t_1]_d$ we write $<A_1,\ldots,A_r>$ to denote the subspace of $\mathbb C[t_0,t_1]_d$ spanned by $\cup A_i$. 
\end{itemize}

\medskip
 \section{Fano congruence of bitangents}\label{s.Fano congruence of bitangents}

We use the same notations as in the introduction. 
\begin{notation}\label{n.family of normalization}Let $d$ be an integer with $d\geq 5$. 
Let  $G={\rm Gr}(2,4)$ be the Grassmannian of lines in $\mathbb P^3$ and let $\mathcal I^{\vee}$ be the dual of the tautological vector bundle $\mathcal I$ on $G$. 
Then  $$\mathcal F=\mathbb P(\Sym^2\mathcal I^{\vee})\times_G\mathbb P(\Sym^{d-4}\mathcal I^{\vee})$$   parametrizes  $P=(\ell, p_1+p_2, q_1+\cdots+q_{d-4})$ with $\ell\in G$ and $ p_i,q_j\in \ell$.  
 
 Let $W$ be the space of hypersurfaces of degree $d$ in $\mathbb P^3$. Then $W\cong\mathbb P^{\frac{d^3+6d^2+11d}{6}}$. Let $\mathcal S$ be the subscheme of $W\times \mathcal F$ parametrizing \begin{align*}(Y,P)\in W\times \mathcal F\end{align*}
such that $P=(\ell, p_1+p_2, q_1+\cdots+q_{d-4})$ with  $Y\cap \ell$ contains $2p_1+2p_2+\sum q_i$ as a subscheme.  
We denote by $\psi:{\mathcal S}\rightarrow W$, $\rho:\mathcal S\rightarrow \mathcal F$ and $\phi: {\mathcal S}\rightarrow G$  the natural projection morphisms:
\begin{center}
\begin{tikzpicture}
	\matrix (m) [matrix of math nodes,row sep=3em,column sep=3em,minimum width=2em]
	{
		& \mathcal S&  & \\
		W& &	\mathcal F& \\ 
		&G & \\};
	\path[-stealth]
	(m-1-2) edge node [left] {$\psi$} (m-2-1)
	(m-1-2) edge node [right] {$\rho$} (m-2-3)
	(m-1-2) edge node [right] {$\phi$} (m-3-2)
	(m-2-3) edge node [right] {} (m-3-2);
\end{tikzpicture}
\end{center}
 For  $Y\in W$   we have    $\psi^{-1}(Y)=S(Y)$ and $\phi(\psi^{-1}(Y))=B(Y)$.   Let us denote by  $\phi_Y: S(Y)\rightarrow B(Y)$ the  restriction of $\phi$. 
\end{notation}

\begin{proposition}\label{p.congruence}
	Let $Y$ be a  hypersurface of degree $d\geq 5$ admitting only isolated singularities. Assume that the dual variety $Y^*\subset (\mathbb P^3)^*$ of $Y$ is an irreducible hypersurface.  Then  $B(Y)$  has dimension 2. 
\end{proposition}
\begin{proof} Let ${\rm Tan}(Y)=S_{2,1^{d-2}}(Y)\subset G$ be the space of lines in $\mathbb P^3$ meeting  $Y$ with multiplicity $\geq 2$ at some point. We first claim that ${\rm Tan}(Y)$ is irreducible of dimension 3. 
	Let $\mathcal T\subset G\times Y$ be the space of $(\ell, p)\in G\times Y$ such that $\ell$ meets $Y$ with multipicity $\geq 2$ at $p$. For each point $p\in Y$, the fiber of the projection morphism $\mathcal T\rightarrow Y$ over $p$ is isomorphic to $\mathbb P^1$ if $Y$ is smooth at $p$, and to $\mathbb P^2$ if $Y$ is singular at $p$.   It follows that  $\mathcal T$ is irreducible of dimension 3 because $Y$ has only isolated singular points. Therefore the image of the projection $\mathcal T\rightarrow G$ which is equal to ${\rm Tan}(Y)$ is irreducible of dimension 3.

	Let $\mathcal K$ be the space of pairs $(p,H)\in Y\times (\mathbb P^3)^*$ such that $p\in H$. Let $\phi_1:\mathcal K\rightarrow Y$ and $\phi_2:\mathcal K\rightarrow (\mathbb P^3)^*$ be the natural projection morphisms. 
	Each fiber of $\phi_1$ is isomorphic to $\mathbb P^2$, and hence $\mathcal K$ is irreducible since $Y$ is irreducible. Therefore for general $H$ outside codimension 2 subscheme of $(\mathbb P^3)^*$, the fiber $\phi_2^{-1}(H)=H\cap Y$ is irreducible (see \cite{BK20}). Therefore when we take a general $H\in Y^*$ so that $H=\hat T_pY$ for some smooth point $p\in Y$, the intersection $C_p:=\hat T_pY\cap Y$ is an irreducible plane curve.

 Suppose that $B(Y)$ has an irreducible  component $D$ of dimension 3.  Since ${\rm Tan}(Y)$ is irreducible of dimension 3 and $D\subset {\rm Tan}(Y)$, $D$ is  equal to ${\rm Tan}(Y)$.  This means that any tangent line of $Y$ is bitangent to $Y$.  
 Take a general point $p\in Y$ so that $C_p=\hat T_pY\cap Y$ is an irreducible  curve in $\hat T_pY\cong \mathbb P^2$. By composing the normalization $\tilde C_p$ and the projection map $\hat T_pY\dashrightarrow \mathbb P^1$ from a point  $p$ we obtain  a finite morphism $\nu:\tilde C_p\rightarrow \mathbb P^1$. Since any tangent line is bitangent by our assumption,  $\nu$ is branched over all points of $\mathbb P^1$ which is impossible. We are done. 
\end{proof}

We remark that  for all $Y$ in the complement of codimension $\geq 2$ subscheme of $W$, the conditions in Proposition~\ref{p.congruence} are satisfied.

\begin{notation}\label{n.coord WP}
	Take a point  $P=(\ell, p_1+p_2, q_1+\cdots+q_{d-4})\in \mathcal F$.  Choose  
	a homogeneous coordinate system $t_0,\ldots,t_3$ on $\mathbb P^3$ so that $\ell$ is defined by $t_2=t_3=0$. 
	There are polynomials $g\in \mathbb C[t_0,t_1]_2$  and $h\in\mathbb C[t_0,t_1]_{d-4}$ such that \begin{center}$\ell\cap(g=0)=p_1+p_2$ and $\ell\cap (h=0)=q_1+\cdots+q_{d-4}$.\end{center} We often write $P=(\ell, [g],[h])$. 
	
	 The fiber  $\mathcal S_P:=\rho^{-1}(P)$ is  the space of  $Y\in W$  such that $P\in S(Y)$. 
	 For $Y\in \mathcal S_P$,  we have  a polynomial $f\in \mathbb C[t_0,\ldots,t_3]_d$    defining  $Y$ such that  
	 \begin{equation}\label{eq.Y}
	 	f=\begin{cases}g^2h+t_2 \bar{g}_{d-1}+t_3 \bar{h}_{d-1}&\mbox{if $\ell\not\subset Y$}\\
	 	t_2 \bar{g}_{d-1}+t_3 \bar{h}_{d-1}&\mbox{if $\ell\subset Y$}\end{cases}
	 \end{equation}
	 for some $\bar{g}_{d-1}=\sum x_{i_0,i_1,i_2}t_0^{i_0}t_1^{i_1}t_1^{i_2}\in\mathbb C[t_0,t_1,t_2]_{d-1}$ and $ \bar h_{d-1}=\sum y_{i_0,i_1,i_2,i_3}t_0^{i_0}t_1^{i_1}t_2^{i_2}t_3^{i_3}\in \mathbb C[t_0,t_1,t_2,t_3]_{d-1}$.  Thus  $\mathcal S_P$ is isomorphic to $\mathbb P^{\dim W-d}$ with homogeneous coordinate ring containing $x_{i_0,i_1,i_2}$ and $y_{i_0,i_1,i_2,i_3}$ as a part of its variables. 
\end{notation}

\begin{proposition}\label{p.general smooth}
	For a general hypersurface  $Y\subset \mathbb P^3$ of degree $d\geq 5$, $S(Y)$ is a smooth irreducible projective surface, and  $\phi_Y:S(Y)\rightarrow B(Y)$ is a finite birational  morphism. Therefore $\phi_Y$ is the normalization of $B(Y)$. 
\end{proposition}
\begin{proof}
Clearly each fiber of $\phi:\mathcal S\rightarrow G$ is isomorphic to $\mathbb P^2\times \mathbb P^{d-4}\times\mathbb P^{\dim W-d}$ which implies  that  $\mathcal S$ is connected and  smooth projective variety of dimension $\dim W+2$. Take a general $Y\in W$. 
	Since $\psi$ is dominant, 
	$S(Y)$ is an irreducible smooth projective surface. The condition $d\geq 5$ implies that general $Y$  contains no  line  and hence $\phi_Y$ is a finite morphism. Furthermore a general bitangent line $\ell$ of  $Y$ satisfies $Y\cap \ell=2p_1+2p_2+q_1+\cdots+q_{d-4}$ for distinct $p_i, q_j\in \ell$.  It follows that $\phi_Y$  is generically one to one. 	
\end{proof}

\medskip
Define $pr_i$ and $\pi_i$ by the projections making the following commutative diagram:

\begin{tikzpicture}
	\matrix (m) [matrix of math nodes,row sep=3em,column sep=4em,minimum width=2em]
	{
		& \mathcal F=\mathbb P(\Sym^2\mathcal I^{\vee})\times_G \mathbb P(\Sym^{d-4}\mathcal I^{\vee}) & \\
		\mathbb P(\Sym^2\mathcal I^{\vee})& &	\mathbb P(\Sym^{d-4}\mathcal I^{\vee})\\
		& G &  \\};
	\path[-stealth]
	(m-1-2) edge node [left] {$pr_1$} (m-2-1)
	edge  node [below] {\ \ \ $\pi$} (m-3-2)
	(m-1-2) edge node [right] {$pr_2$} (m-2-3)
	(m-2-1) edge node [left] {$\pi_1$} (m-3-2)
	(m-2-3) edge node [right] {$\pi_2$} (m-3-2);
\end{tikzpicture}

\noindent
Set $L:=pr_1^*\mathcal O(1)$ and $M:=pr_2^*\mathcal O(1)$. Then $L^{-1}$  and $M^{-1}$ are subbundles of $\pi^*\Sym^2\mathcal I^{\vee}$ and $\pi^*\Sym^{d-4}\mathcal I^{\vee}$ respectively. So we get a natural inclusion:
$$ L^{-2}\otimes M^{-1}\hookrightarrow \pi^*\Sym^d\mathcal I^{\vee}$$
 Define a vector bundle $R$ on $\mathcal F$   as the cokernel of the inclusion above. 
 
For each $Y\in W$ containing no line, its defining equation gives a section  $\mathcal O_{\mathcal F}\rightarrow \pi^*\Sym^d\mathcal I^{\vee}$. Composing with the quotient map we obtain a  map $\mathcal O_{\mathcal F}\rightarrow R$ which vanishes on $S(Y)$. Thus $S(Y)$ is defined by zeroes of a section of $R$ (cf. \cite{Ful}). Since $R$ has rank $d$, this implies that if $S(Y)$ is a projective surface then it has at worst locally complete intersection singularities.

\begin{corollary}
Let  $Y\subset \mathbb P^3$ be as in Proposition \ref{p.congruence}. If $Y$ contains no line, then  $[S(Y)]=c_d(R)$ in the Chow ring $ch(\mathcal F)$. 
\end{corollary}

\section{Criterion for smoothness of Fano congruence of bitangents}
In this section we present a criterion for smoothness of Fano congruence of bitangents. 
 We introduce the following notation for convenience. 
\begin{notation}\label{n.eq}	
	In Notation \ref{n.coord WP},  
  choose $r_1,r_2\in\mathbb C[t_0,t_1]_2$ and $s_1,\ldots,s_{d-4}\in\mathbb C[t_0,t_1]_{d-4}$ such that  \begin{center}	$<g,r_1,r_2>=\mathbb C[t_0,t_1]_2$ and
	$<h,s_1,\ldots,s_{d-4}>=\mathbb C[t_0,t_1]_{d-4}.$\end{center}
Set $$A_P:=<g^2h, gr_1h, gr_2h, g^2s_1,\ldots,g^2s_{d-4}>=g<gh, r_1h, r_2h, gs_1,\ldots,gs_{d-4}>.$$
\end{notation}

\bigskip 
Next is a technical lemma which will be used to show the smoothness of Fano congruence of bitangents. 
\begin{lemma}\label{l.key lemma}  In the situation of Notations \ref{n.coord WP} and \ref{n.eq},  take $Y\in\mathcal S_P$ with  $\ell\not\subset Y$ and set $g_{d-1}:=\bar g_{d-1}(t_0,t_1,0)$ and $h_{d-1}:=\bar h_{d-1}(t_0,t_1,0,0)\in\mathbb C[t_0,t_1]_{d-1}$ where $\bar g_{d-1}$ and $\bar h_{d-1}$ are the polynomials associated to $Y$ as in Notation \ref{n.coord WP}.  
	Let  $$B_Y:= <g^2h, t_0g_{d-1}, t_1g_{d-1},t_0h_{d-1},t_1h_{d-1}>\subset \mathbb C[t_0,t_1]_d.$$ 
	Then if $\dim A_P=d-1$, then $\phi_Y$ is immersed at $P$, and   if $<A_P,B_Y>=\mathbb C [t_0,t_1]_d$ then  $S(Y)$ is a surface smooth at $P$.
\end{lemma}

\begin{proof}
	Define a morphism  $$\varphi:\mathcal F\rightarrow \mathbb P(\Sym^d\mathcal I^{\vee})$$ by $$P'=(\ell',p_1'+p_2',q_1'+\cdots+q_{d-4}')\mapsto (\ell', 2(p_1'+p_2')+q_1'+\cdots+q_{d-4}').$$

	 Let $U\cong \mathbb C^4_{a_0,a_1,b_0,b_1}\subset G$ be the affine neighborhood of $\ell$ parametrizing lines defined by $t_2=a_0t_0+a_1t_1$, and 
	$t_3=b_0t_0+b_1t_1$. 
 The deformation of $P=(\ell,[g],[h])$ in $\mathcal F$ is given by 
	\begin{center}
		 $t_2=a_0t_0+a_1t_1$, 
		$t_3=b_0t_0+b_1t_1$, $g+u_1r_1+u_2r_2$ and $h+v_1s_1+\cdots+v_{d-4}s_{d-4}$
	\end{center}where $a_i, b_i, u_i, v_i\in\mathbb C$ and hence there is a natural identification  $$T_{P}\mathcal F=\mathbb C^4_{a_0,a_1,b_0,b_1}\times \mathbb C^2_{u_1,u_2}\times \mathbb C^{d-2}_{v_1,\ldots,v_{d-4}}.$$ 
We have a trivialization of the projective bundle $\mathbb P(\Sym^{d}\mathcal I^{\vee})|_U\cong U\times \mathbb P(\mathbb C[t_0,t_1])_d)$ and hence   $$T_{\phi(P)}\mathbb P(\Sym^d\mathcal I^{\vee})=\mathbb C^{4}_{a_0,a_1,b_0,b_1}\times T_{[g^2h]}\mathbb P(\mathbb C[t_0,t_1]_d).$$
The terms of the order $\le 1$ of the equation    
	\begin{align*}&(g+u_1r_1+u_2r_2)^2(h+v_1s_1+\cdots+v_{d-4}s_{d-4})\\&=g^2h+2u_1gr_1h+2u_2gr_2h+v_1g^2s_1+\cdots+v_{d-4}g^2s_{d-4}+\{\mbox{order $\geq 2$ in $u_i, v_j$}\},\end{align*}
	 is the first order deformation of the equation $g^2h$. Therefore  $d\phi(T_P\mathcal F)$ is 
	\begin{equation}\label{eq.ImagF}\{((a_0,a_1,b_0,b_1),[g^2h+2u_1gr_1h+2u_2gr_2h+v_1g^2s_1+\cdots+v_{d-4}g^2s_{d-4}])\ |\ a_i,b_j\in\mathbb C\}\end{equation} which is equal to the subspace   $\mathbb C^4_{a_0,a_1,b_0,b_1}\times T_{[g^2h]}\mathbb P(A_P)$ of $\mathbb C^{4}_{a_0,a_1,b_0,b_1}\times T_{[g^2h]}\mathbb P(\mathbb C[t_0,t_1]_d)$. 	
Therefore 	
\begin{equation}\label{eq.F} \dim d\varphi( T_P\mathcal F)=4+(\dim A_P-1).\end{equation}

Since $\ell\not\subset Y$, we have a local section  $G_Y$  of the projective  bundle $\mathbb P(\Sym^d\mathcal I^{\vee})\rightarrow G$ defined by 
$\ell'\mapsto (\ell', Y\cap \ell')$  in some neighborhood $U'\subset U$ of $\ell$.   Under the isomorphsim $\mathbb P(Sym^{d}\mathcal I^{\vee})|_U\cong U\times \mathbb P(\mathbb C[t_0,t_1])_d)$,  the restriction $G_Y|_{U'}$ is given by  $$(a_0,a_1,b_0,b_1)\mapsto ((a_0,a_1,b_0,b_1), [f(t_0,t_1, a_0t_0+a_1t_1,b_0t_0+b_1t_1)])\in U\times \mathbb P(\mathbb C[t_0,t_1]_d).$$ 
From \begin{align*}&f(t_0,t_1, a_0t_0+a_1t_1,b_0t_0+b_1t_1)\\&=g^2h+a_0t_0g_{d-1}+a_1t_1g_{d-1}+b_0t_0h_{d-1}+b_1t_1h_{d-1}+\mbox{\{order $\geq 2$ in $a_i,b_j$\}},\end{align*}
 it follows that the  tangent space $T_{\varphi(P)}G_Y$  is equal to  
\begin{equation}\label{eq.G}\{((a_0,a_1,b_0,b_1),[g^2h+a_0t_0g_{d-1}+a_1t_1g_{d-1}+b_0t_0h_{d-1}+b_1t_1h_{d-1}])\ |\ a_i, b_j\in \mathbb C\}\end{equation} which is contained in the subspace $\mathbb C^4_{a_0,a_1,b_0,b_1}\times T_{[g^2h]}\mathbb P( B_Y)$ of $\mathbb C^4_{a_0,a_1,b_0,b_1}\times T_{[g^2h]}\mathbb P(\mathbb C[t_0,t_1]_d)$ 
From this and (\ref{eq.F}), we have 	

\begin{align}\label{eq.FG}\dim d\varphi(T_P\mathcal F)\cap T_{\varphi(P)}G_Y&=4-(\dim B_Y-\dim A_P\cap B_Y)=4-(\dim <A_P,B_Y>-\dim A_P).\end{align}
Since	$\varphi(S(Y))=\varphi(\mathcal F)\cap G_Y$, we have $$\dim d\varphi( T_PS(Y))\leq \dim d\varphi (T_P\mathcal F)\cap T_{\varphi(P)}G_Y=4-(\dim <A_P,B_Y>-\dim A_P).$$
From this and (\ref{eq.ImagF}), we obtain  \begin{align*}2\leq \dim T_PS(Y)&\leq \dim T_P\mathcal F-\dim d\varphi( T_P\mathcal F)+ \dim d\varphi(T_PS(Y))\\&
\leq (d+2)-(3+\dim A_P)+(4-\dim <A_P,B_Y>-\dim A_P)\\&=	 d+3-\dim <A_P,B_Y>.\end{align*}  
Therefore if $\dim <A_P,B_Y>=d+1$ then $\dim T_PS(Y)=2$, and hence we get the first statement.  (\ref{eq.F}) says that if $\dim A_P=d-1$ then $\varphi$ is immersed at $P$ and thus so is $\phi_Y$ because $\phi_Y$ is  locally  a restriction of $\varphi$ at $P$.  We are done. 
\end{proof}

Next theorem  describes the smooth locus of Fano congruence of bitangents to arbitrary $Y$. 

\begin{theorem}\label{t.smooth}Let $Y$ is a hypersurface in $\mathbb P^3$ of degree $d\geq 5$. Let $P=(\ell, p_1+p_2, q_1+\cdots+q_{d-4})$ be a point of $ S(Y)$ such that  $\ell\not\subset Y$, $\{p_1,p_2\}\cap \{q_1,\ldots,q_{d-4}\}=\emptyset$. Then $\phi_Y:S(Y)\rightarrow B(Y)$ is immersed at $P$. If $Y$ is smooth at $p_1$ and $p_2$, then  $S(Y)$ is a smooth surface at $P$. 
\end{theorem}

\begin{proof}
Take a point  $P=(\ell, p_1+p_2, q_1+\cdots+q_{d-4})\in\mathcal F$ such that  $\{p_1,p_2\}\cap \{q_1,\ldots,q_{d-4}\}=\emptyset$. Choose homogeneous coordinates $t_0,\ldots,t_3$ in $\mathbb P^3$ and $g, h, r_i, s_j\in\mathbb C[t_0,t_1]$ so that 
$\ell=(t_2=t_3=0)$ and $P=(\ell,[g],[h])$ as in Notations \ref{n.coord WP} and \ref{n.eq}. By the assumption $\{p_1,p_2\}\cap \{q_1,\ldots,q_{d-4}\}=\emptyset$  we have 
	$$A_P=g\mathbb C[t_0,t_1]_{d-2}.$$ 
Therefore by Lemma \ref{l.key lemma},  the morphism  $\phi_Y$ is immersed at $P$ for any $Y$ with $\ell\not\subset Y$.

Let $Y\in S_P$ be a hypersurface of degree $d\geq 5$ such that $\ell\not\subset Y$ and $Y$ is  smooth at $p_1$ and $p_2$.  	
 Let $x_{i_0,i_1,i_2}$, $y_{i_0,i_1,i_2,i_3}$, $g_{d-1}$ and  $h_{d-1}$  be  as in Notation \ref{n.coord WP} and Lemma \ref{l.key lemma}. Then 
 $$g_{d-1}=x_{d-1,0,0}t_0^{d-1}+x_{d-2,1,0,0}t_0^{d-2}t_1+\cdots+x_{0,d-1,0}t_1^{d-1}$$ and $$h_{d-1}=y_{d-1,0,0,0}t_0^{d-1}+y_{d-2,1,0,0}t_0^{d-2}t_1+\cdots+y_{0,d-1,0,0}t_1^{d-1}.$$
 
\medskip
{\bf Case 1:} Assume $p_1\neq p_2$. We may set    
	$g=t_0t_1$ so that   $p_1=(1:0:0:0)$, $p_2=(0:1:0:0)$, and	$A_P=t_0t_1\mathbb C[t_0,t_1]_{d-2}.$
By modulo $A_P$, we have the following   equivalences  
\begin{align*}t_0g_{d-1}&\equiv x_{d-1,0,0}t_0^d,\\
	t_1g_{d-1}&\equiv x_{0,d-1,0}t_1^d,\\
	t_0h_{d-1}&\equiv y_{d-1,0,0,0}t_0^d, \\
	t_1h_{d-1}&\equiv y_{0,d-1,0,0}t_1^d.
\end{align*} Therefore  $<A_P,B_Y>=\mathbb C[t_0,t_1]_d$ if and only if the rank of the next matrix $M_Y$ is equal to 2;

 $$M_Y=\begin{pmatrix} x_{d-1,0,0}&0\\
	0&x_{0,d-1,0}\\
	y_{d-1,0,0,0}&0\\
	0&y_{0,d-1,0,0}
\end{pmatrix}.$$
Since $(1:0:0:0)$ and $(0:1:0:0)$ are smooth points of $Y$,   we have 
\begin{center}$x_{d-1,0,0}\neq 0$ or $y_{d-1,0,0}\neq0$\end{center}
and \begin{center}$x_{0,d-1,0}\neq 0$ or $y_{0,d-1,0,0}\neq0$.\end{center}	
Therefore $M_Y$ has rank 2 which implies that $S(Y)$ is smooth at $P$ by Lemma \ref{l.key lemma}. 
	
\medskip
{\bf Case 2:}	Suppose that   $p_1=p_2$. We may assume that $g=t_0^2$ so that $p_1=p_2=(0:1:0:0)$ and  $A_P=t_0^2\mathbb C[t_0,t_1]_{d-2}$. By modulo $A_P$,  we have the following  equivalences:
	\begin{align*}t_0g_{d-1}&\equiv x_{0,d-1,0}t_0t_1^{d-1},\\
		t_1g_{d-1}&\equiv x_{1,d-2,0}t_0t_1^{d-1}+x_{0,d-1,0}t_1^d,\\
		t_0h_{d-1}&\equiv y_{0,d-1,0,0}t_0t_1^{d-1}, \\
		t_1h_{d-1}&\equiv y_{1,d-2,0,0}t_0t_1^{d-1}+y_{0,d-1,0,0}t_1^d.
	\end{align*}
 Therefore $<A_P,B_Y>=\mathbb C[t_0,t_1]_d$ if and only if the rank of the matrix $$M_Y=\begin{pmatrix} x_{0,d-1,0}& 0\\
		x_{1,d-2,0}& x_{0,d-1,0}\\
		y_{0,d-1,0,0} &0\\
		y_{1,d-2,0,0}&y_{0,d-1,0,0}
	\end{pmatrix}$$is 2. 
	Since $Y$ is smooth at $(0:1:0:0)$,	we have $x_{0,d-1,0}\neq 0$ or $y_{0,d-1,0,0}\neq 0$ which implies that  $M_Y$ has rank 2. 	Therefore $S(Y)$ is a smooth surface locally at $P$ by  Lemma~\ref{l.key lemma}.
\end{proof}

\begin{remark}Assume that $Y$ is smooth and contains no line. 
Theorem~\ref{t.smooth} says that  $S(Y)\setminus \phi_Y^{-1}(S_{3,2,1^{d-5}}(Y))$ is either empty  or a smooth surface such that the restriction $\phi_Y|_{S(Y)\setminus \phi_Y^{-1}(S_{3,2,1^{d-5}}(Y))}$ is an immersion. We remark that $\phi_Y$ is one to one on the outside of $\phi^{-1}(S_{2^3,1^{d-6}}(Y))$.  Thus  $B(Y)\setminus \{S_{2^3,1^{d-6}}(Y)\cup S_{3,2,1^{d-5}}(Y)\}$ is either empty  or  a  smooth surface. This generalizes the result in \cite{T80} on quartic surface. In \cite{ABT01}, there are some results on singularities of $B(Y)$  for general $Y$.  
\end{remark}

\section{Fano congruence of bitangents in Lefschetz pencil of hypersurfaces}
In this section we are going to consider the singularities of the Fano congruences of bitangents  in a general Lefschetz pencil.   We remark that since $d\geq 5$,  the space of  $Y\in W$  containing some line has codimension  $\geq 2$ in $W$. So if we take a general line $L\subset W$ then any $Y\in L$ contains no line. 

\begin{proposition}\label{p.isolated singularity}
	Let  $L$ be a general line in $W$ and let $Y\in L$. The set of singular points $P=(\ell, p_1+p_2, q_1+\cdots+q_{d-4})$ of $S(Y)$ such that $Y$ is smooth at $p_1$ and $p_2$ is finite.  	
\end{proposition}

\begin{proof}
Let $\mathcal U$ be the subscheme of $\mathcal S$ parametrizing $(Y,P)$ such that $P=(\ell,p_1+p_2, q_1+\cdots+q_{d-4})\in S(Y)$ with $\{p_1,p_2\}\cap \{q_1,\ldots,q_{d-4}\}\neq\emptyset$, $\ell\not\subset Y$ and $S(Y)$ is singular at $P$. Let us denote by $\psi^s:\mathcal U\rightarrow W$ and $\rho^s:\mathcal U\rightarrow \mathcal F$ the respective restrictions of  the projections $\psi:\mathcal S\rightarrow W$ and  $\rho:\mathcal S\rightarrow \mathcal F$. 
 For the proof, it is enough to show that $\dim \mathcal U\leq \dim W-1$. In fact this implies that if we take a general line $L\subset W$ then each inverse image $(\psi^s)^{-1}(Y)$, $Y\in L$, consists of at most finite elements.  From this and  Theorem \ref{t.smooth} we get the proof.  
 
Take  a point $P=(\ell, p_1+p_2,q_1+\cdots+q_{d-4})\in\mathcal F$ with $\{p_1,p_2\}\cap \{q_1,\ldots,q_{d-4}\}\neq\emptyset$.  Choose a homogeneous coordinates $t_0,\ldots,t_3$ on $\mathbb P^3$ and $g,h, r_i,s_j\in\mathbb C[t_0,t_1]$ so that $\ell=(t_2=t_3=0)$, $P=(\ell, [g],[h])$,  $<g,r_1,r_2>=\mathbb C[t_0,t_1]_2$ and $<h,s_1,\ldots,s_{d-4}>=\mathbb C[t_0,t_1]_{d-4}$ as in Notations~\ref{n.coord WP} and \ref{n.eq}. 

From the arguments in Notation \ref{n.coord WP} we can see that the open subset $\mathcal S_P^o\subset \mathcal S_P$ parametrizing $Y\in W$ such that $\ell\not\subset Y$ and  $P\in S(Y)$ is isomorphic to $\mathbb C^{\dim W-d}$. Clearly, $(\rho^s)^{-1}(P)\subset \mathcal S_P^o$. For $Y\in \mathcal S_P^o$, let $g_{d-1}, h_{d-1}, x_{i_0,i_1,i_2},y_{i_0,i_1,i_2,i_3}$ be as in Lemma \ref{l.key lemma} and  Notation \ref{n.coord WP}. Then 
$$ g_{d-1}=x_{d-1,0,0}t_0^{d-1}+x_{d-2,1,0}t_0^{d-2}t_1+\cdots+x_{0,d-1,0,0}t_1^{d-1}$$ and $$h_{d-1}=y_{d-1,0,0,0}t_0^{d-1}+y_{d-2,1,0,0}t_0^{d-2}t_1+\cdots+y_{0,d-1,0,0}t_1^{d-1}.$$  The affine coordinate ring $R(\mathcal S_P^o)$ of $\mathcal S_P^o$  is a polynomial ring  in variables   
$x_{i_0,i_1,i_2},y_{i_0,i_1,i_2,i_3}$.  In order to find some upper bound of the dimension of  $(\rho^s)^{-1}(P)\subset \mathcal S_P^o$  we are going to describe an ideal in $R(\mathcal S_P^o)$ which defines a subscheme of $\mathcal S_P^0$ containing  $(\rho^s)^{-1}(P)$.   

\medskip
Up to renumbering we may assume that  $p_1=q_1$. 

\medskip
{\bf Case 1: } Suppose that $p_1\neq p_2$. 
 We can set  
\begin{center}  $g=t_0t_1$, $r_1=t_0^2$, $r_2=t_1^2$, $h=t_0h'$\end{center} for some  $h'\in\mathbb C[t_0,t_1]_{d-5}$ so that $p_1=(0:1:0:0)$ and $p_2=(1:0:0:0)$. 
Then 	\begin{align*}A_P&=t_0^2t_1<t_1\mathbb C[t_0,t_1]_{d-4},t_0^2h'>.\end{align*} 

\medskip
{\bf Case 1-1}: Suppose $p_2\not\in\{q_1,\ldots,q_{d-4}\}$. Then $t_1\nmid h'$ and hence   $$A_P=<t_0^2t_1^2\mathbb C[t_0,t_1]_{d-5}, t_0^{d-1}t_1>.$$
 By modulo $ A_P$, 
we have the following  equivalences:
\begin{align*}t_0g_{d-1}&\equiv x_{d-1,0,0}t_0^d+x_{0,d-1,0}t_0t_1^{d-1},\\
	t_1g_{d-1}&\equiv x_{1,d-2,0}t_0t_1^{d-1}+x_{0,d-1,0}t_1^d,\\
	t_0h_{d-1}&\equiv y_{d-1,0,0,0}t_0^d+y_{0,d-1,0,0}t_0t_1^{d-1}, \\
	t_1h_{d-1}&\equiv y_{1,d-2,0,0}t_0t_1^{d-1}+y_{0,d-1,0,0}t_1^d. 
\end{align*}
Therefore for $B_Y=<g^2h, t_0g_{d-1}, t_1g_{d-1},t_0h_{d-1},t_1h_{d-1}>$, we have   $<A_P,B_Y>=\mathbb C[t_0,t_1]_d$ if and only if   the next $4\times 3$ matrix	 
$$M_Y=\begin{pmatrix}  x_{d-1,0,0}& x_{0,d-1,0}&0\\
	0&x_{1,d-2,0}& x_{0,d-1,0}\\
	y_{d-1,0,0,0}& y_{0,d-1,0,0}&0\\
	0&y_{1,d-2,0,0}& y_{0,d-1,0,0}
\end{pmatrix}$$
has rank $3$. The subscheme of $S_P^o$ defined by the  ideal of $R(\mathcal S_P^o)$ generated by $3\times 3$ minors of $M_Y$ has codimension 2 in $S_P$. Therefore  $\dim (\rho^s)^{-1}(P)\leq \dim \mathcal S_P^o-2=\dim W-d-2.$ Let $C_{1,1}$ be the space of $P$ in Case 1-1. Then $C_{1,1}$ has dimension $d+1$, and hence $\dim (\rho^s)^{-1}(C_{1,1})\leq \dim W-1$.

\medskip
{\bf Case 1-2:} Suppose $p_2\in\{q_1,\ldots,q_{d-4}\}$.   Then  $t_1\mid h'$ and hence we have  $$A_P=t_0^2t_1^2\mathbb C[t_0, t_1]_{d-4}.$$ By modulo  $A_P$, we have the following equivalences 
\begin{align*}t_0g_{d-1}&\equiv x_{d-1,0,0}t_0^d+x_{d-2,1,0}t_0^{d-1}t_1+x_{0,d-1,0}t_0t_1^{d-1},\\
	t_1g_{d-1}&\equiv x_{d-1,0,0}t_0^{d-1}t_1+x_{1,d-2,0}t_0t_1^{d-1}+x_{0,d-1,0}t_1^d,\\
	t_0h_{d-1}&\equiv y_{d-1,0,0,0}t_0^d+y_{d-2,1,0,0}t_0^{d-1}t_1+y_{0,d-1,0,0}t_0t_1^{d-1}, \\
	t_1h_{d-1}&\equiv  y_{d-1,0,0,0}t_0^{d-1}t_1+y_{1,d-2,0,0}t_0t_1^{d-1}+y_{0,d-1,0,0}t_1^d. 
\end{align*}
It follows that  $\dim <A_P,B_Y>=\mathbb C[t_0,t_1]_d$ if and only if  the next $4\times 3$ matrix	 
$$M_Y=\begin{pmatrix}  x_{d-1,0,0}& x_{d-2,1,0}&x_{0,d-1,0}&0\\
	0&x_{d-1,0,0}&x_{1,d-2,0}& x_{0,d-1,0}\\
	y_{d-1,0,0,0}&y_{d-2,1,0,0}& y_{0,d-1,0,0}&0\\
	0&y_{d-1,0,0,0}&y_{1,d-2,0,0}& y_{0,d-1,0,0}
\end{pmatrix}$$
has rank 4. 
This shows  $(\rho^s)^{-1}(P)$  has dimension $\leq  \dim \mathcal S_P^o-1=\dim W-d-1.$ Let $C_{1,2}$ be the space of $P$ in Case 1-2. Then $\dim C_{1,2}=d$ which implies that $\dim(\rho^s)^{-1}(C_{1,2})\leq \dim W-1$.

\medskip
{\bf Case 2:} Assume that $p_1=p_2$.  We can set  
\begin{center}  $g=t_0^2$, $r_1=t_0t_1$, $r_2=t_1^2$, $h=t_0h'$\end{center} for some  $h'\in\mathbb C[t_0,t_1]_{d-5}$ so that $p_1=p_2=(0:1:0:0)$ and 
	\begin{align*}A_P&=t_0^3<t_0\mathbb C[t_0,t_1]_{d-4}, t_1^2h'>.\end{align*} 

\medskip
{\bf Case 2-1:} Assume that  $t_0\nmid h'$ so that $p_1=p_2=q_1\neq q_i$ for all $i\geq 2$ up to renumbering.  We have  
$$A_P=<t_0^4\mathbb C[t_0,t_1]_{d-4}, t_0^3t_1^{d-3}>.$$ 
By the same arguments as above, $\dim <A_P,B_Y>=\mathbb C[t_0,t_1]_d$ if and only if  $4\times 3$ matrix	 
$$M_Y=\begin{pmatrix}  x_{1,d-2,0}& x_{0,d-1,0}&0\\
	x_{2,d-3,0}&x_{1,d-2,0}& x_{0,d-1,0}\\
	y_{1,d-2,0,0}& y_{0,d-1,0,0}&0\\
	y_{2,d-3,0,0}&y_{1,d-2,0,0}& y_{0,d-1,0,0}
\end{pmatrix}$$
has rank 3. The subscheme of $\mathcal S_P^o$ defined by the ideal  of $R(\mathcal S_P^o)$ generated by $3\times 3$ minors of $M_Y$ has codimension 2 in $\mathcal S_P^o$. Therefore  $\dim(\rho^s)^{-1}(P)\leq\dim \mathcal S_P^o-2=\dim W-d-2$. The space $C_{2,1}$ of $P$ in Case 2-1 has dimension $d$, and thus $\dim (\rho^s)^{-1}(C_{2,1})\leq  \dim W-2$. 

\medskip
{\bf Case 2-2:}
Assume that $t_0\mid h'$. In this case  $p_1=p_2=q_1=q_i$ for some $i\geq 2$. It is easy to check  $A_P=t_0^4\mathbb C[t_0,t_1]_{d-4}$ 
and hence $<A_P,B_Y>=\mathbb C[t_0,t_1]_d$ if and only if the next matrix $M_Y$ has rank 4:
$$M_Y=\begin{pmatrix}  x_{2,d-3,0}&x_{1,d-2,0}& x_{0,d-1,0}&0\\
	x_{3,d-4,0}&x_{2,d-3,0}&x_{1,d-2,0}& x_{0,d-1,0}\\
	y_{2,d-3,0,0}&y_{1,d-2,0,0}& y_{0,d-1,0,0}&0\\
	y_{3,d-4,0,0}&y_{2,d-3,0,0}&y_{1,d-2,0,0}& y_{,d-1,0,0}
\end{pmatrix}.$$
This shows that $\dim (\rho^s)^{-1}(P)\leq \dim \mathcal S_P^o-1=\dim W-d-1$. Since  the space $C_{2,2}$ of $P$ in Case 2-2 has dimension $d-1$, it follows that $\dim (\rho^s)^{-1}(C_{2,2})\leq \dim W-2$. 
\end{proof}

\begin{theorem}\label{t.singular locus}Let $L\subset W$ be a general line and let $Y\in L$. Assume that $S(Y)$ is singular at  some point $P=(\ell, p_1+q_2, q_1+\cdots+q_{d-4})$ such that $Y$ is smooth at $p_1$ and $p_2$.  Then $p_1\neq p_2$, $\{p_1,p_2\}\cap \{q_1,\ldots,p_{d-4}\}\neq \emptyset$ and $q_1\ldots,q_{d-4}$ are distinct. Moreover if $p_1\in \{q_1,\ldots,q_{d-4}\}$ and $p_2\not\in\{q_1,\ldots,q_{d-4}\}$, then $\hat T_{p_1}Y=\hat T_{p_2}Y$ and $\hat T_{p_1}Y\cap Y$ is a cuspical plane curve. 
\end{theorem}
\begin{proof}Use the same  notation in the proof of Proposition \ref{p.isolated singularity}. 
	
	Assume that $S(Y)$ is singular at  some point $P=(\ell, p_1+q_2, q_1+\cdots+q_{d-4})$ such that $Y$ is smooth at $p_1$ and $p_2$. By Theorem \ref{t.smooth} we have   $\{p_1,p_2\}\cap \{q_1,\ldots,q_{d-4}\}\neq\emptyset.$ 
	By proof of Proposition \ref{p.isolated singularity}  we proved that $\dim \psi^s((\rho^s)^{-1}(C_{2,i}))\leq \dim W-2$  so that $L\cap \psi^s((\rho^s)^{-1}(C_{2,i}))=\emptyset $ for $i=1,2$. Therefore   $P\in C_{1,1}\cup C_{1,2}$ and hence $p_1\neq p_2$. 
 Let $D_{1,i}\subset C_{1,i}$ be the space of $P\in C_{1,i}$ such that $q_j$ are not distinct. Then $\dim D_{1,1}=d$ and $\dim D_{1,2}=d-1$. By the dimension bound of the fiber of $\rho^s$ in the proof of Proposition \ref{p.isolated singularity} it follows that $\dim \psi^s((\rho^s)^{-1}(D_{1,i}))\leq \dim W-2$ so that 
$ \psi^s((\rho^s)^{-1}(D_{1,i}))\cap L=\emptyset $. So we get the first statement.

 \medskip
 Assume that $p_1\in\{q_1,\ldots,q_{d-4}\}$ and $p_2\not\in \{q_1,\ldots,q_{d-4}\}$ so that $P\in C_{1,1}$. In the situation of Case 1-1 in the proof of Proposition \ref{p.isolated singularity}, we proved that $S(Y)$ is singular at $P$  only when 
 $$M_Y=\begin{pmatrix}  x_{d-1,0,0}& x_{0,d-1,0}&0\\
 	0&x_{1,d-2,0}& x_{0,d-1,0}\\
 	y_{d-1,0,0,0}& y_{0,d-1,0,0}&0\\
 	0&y_{1,d-2,0,0}& y_{0,d-1,0,0}
 \end{pmatrix}$$ has rank $<3$. Therefore 
\begin{equation}\label{eq.1}x_{d-1,0,0}=y_{d-1,0,0,0}=0,\end{equation} \begin{equation}\label{eq.2}x_{0,d-1,0}=y_{0,d-1,0}= 0,\end{equation}
or
\begin{equation}\label{eq.3}x_{0,d-1,0}y_{1,d-2,0,0} - x_{1,d-2,0}y_{0,d-1,0,0}=x_{0,d-1,0}y_{d-1,0,0,0} - x_{d-1,0,0}y_{0,d-1,0,0}\end{equation}
\begin{equation*}	=x_{1,d-2,0}y_{d-1,0,0,0} - x_{d-1,0,0}y_{1,d-2,0,0}=0. \end{equation*}
Since $Y$ is smooth at $p_1=(0:1:0:0)$ and $p_2=(1:0:0:0)$, (\ref{eq.1}) and (\ref{eq.2}) do not hold. Therefore (\ref{eq.3}) should be satisfied. This means that the planes determined by \begin{center}$t_2x_{0,d-1,0}+t_3y_{0,d-1,0,0}=0$, $t_2x_{d-1,0,0}+t_3h_{d-1,0,0,0}=0$, and $t_2x_{1,d-2,0}+t_3y_{1,d-2,0,0}=0$\end{center} are equal. 
It is easy to check that the embedded tangent spaces of $Y$ at $p_1$ and $p_2$ are defined by  \begin{center} $\hat T_{p_1}Y=
(t_2x_{0,d-1,0}+t_3y_{0,d-1,0,0}=0)$ and   
 $\hat T_{p_2}Y=(t_2x_{d-1,0,0}+t_3h_{d-1,0,0,0}=0).$
 \end{center} Furthermore  the emdedded tangent space of the intersection curve $\hat T_{p_1}Y\cap Y$ is given by 
$$t_2x_{0,d-1,0}+t_3y_{0,d-1,0,0}=0$$ and   $$t_0(t_2x_{1,d-2,0}+t_3y_{1,d-2,0,0})+t_2^2x_{0,d-2,1}+t_2t_3y_{0,d-2,1,0}+t_3^2y_{0,d-2,0,1}=0.$$
Therefore $\hat T_{p_1}Y=\hat T_{p_2}Y$ and $\hat T_{p_1}Y\cap Y$ has cuspidal singularity at $p_1$. 
\end{proof}

We now want to consider the Fano congruence of a hypersurface with a single node as singularities. 
\begin{theorem}\label{t.singular hypersurface} Let $Y$ be a general singular hypersurface of degree $d$ with a node  $p$. Let $\Gamma_Y$  be the space of $P=(\ell, p_1+p_2, q_1+\cdots+q_{d-4})\in S(Y)$ such that $p\in \{p_1,p_2\}$. Then $\Gamma_Y$  is a smooth  irreducible projective curve and  $S(Y)$ is a surface with a double point of rank 2 along $\Gamma_Y$. 
\end{theorem}
\begin{proof}	
	Let $\mathcal  G$ be the space of the pairs $(p,P)\in\mathbb P^3\times\mathcal   F$  such that $P= (\ell, p_1+p_2,q_1+\cdots+q_{d-4})$ with  $p\in\{p_1,p_2\}$. It is easy to check that  $\mathcal G$ is smooth irreducible of  dimension $d+2$. 
	Let $\mathcal H$ be the space of $(Y, (p,P))\in W\times\mathcal G$ such that  $P\in S(Y)$ and $Y$ is singular at $p$. 
	Let us denote by $\psi_1:\mathcal H\rightarrow W$ and $\psi_2:\mathcal H\rightarrow \mathcal G$ the  natural projection morphisms.

	Take a point   $Q=(p,P)\in \mathcal G$ where $P=(\ell, p_1+p_2,q_1+\cdots+q_{d-4})$. Choose a homogeneous coordinate system $t_0,\ldots,t_3$ on $\mathbb P^3$ and $g,h\in\mathbb C[t_0,t_1]$ so that $p=(1:0:0:0)$, $\ell=(t_2=t_3=0)$, and $P=(\ell, [g],[h])$.  
 
 For $Y\in\psi_1 \psi_2^{-1}(Q)$, let $x_{i_0,i_1,i_2}$ and $y_{i_0,i_1,i_2,i_3}$ be as in Notation~\ref{n.coord WP}.  Since $Y$ is singular at $p=(1:0:0:0)$,  $x_{d-1,0,0}=y_{d-1,0,0,0}=0$ which shows that the fiber $\psi_2^{-1}(Q)$  is  a codimension 2 linear subspace of $\mathcal S_P\cong\mathbb P^{\dim W-d}$.  Therefore $\mathcal H$ is smooth and irreducible projective variety of dimension $\dim W$, and hence for general $Y\in \psi_1(\psi_2^{-1}(Q))$, $\psi_1^{-1}(Y)$ which is isomorphic to $\Gamma_Y$ is a smooth irreducible projective curve.

	Let $\mathcal Z \subset\mathcal H$ be a subscheme such that  $\mathcal H\setminus \mathcal Z$ is the space of $( Y,(p,P))\in\mathcal H$ such that $S(Y)$ is a surface with a double point of rank $\geq 2$ at $P$. 
	  For the proof of the last statement  it is enough to show that  $\dim \mathcal Z\leq \dim W-2$. In fact this implies that $\dim \psi_1(\mathcal Z)\leq \dim W-2$. So if we take a general  $Y\in W\setminus \psi_1(\mathcal Z)$ with a single node, $S(Y)$ has a double point  of rank $\geq 2$ along $\Gamma_Y$. Since $S(Y)$ is singular along a smooth curve $\Gamma_Y$, the rank of its double points should be 2. This gives the last statement in our theorem.

	We now want to show $\dim \mathcal Z\leq \dim W-2$. 
	
	Since $\dim \mathcal G=d+2$, it is enough to show that
	 for each $Q\in\mathcal G$ outside codimension 2 subscheme  of $\mathcal G$, 
	   $\mathcal Z_Q:=\mathcal Z\cap\psi_2^{-1}(Q)$ has codimension $\geq 2$ in  $\psi_2^{-1}(Q)\cong \mathbb P^{\dim W-d-2}$. Thus we only need to consider $Q$  in  the following two cases 1 and 2. Indeed the space of $Q$ not belonging to one of these two cases has codimension $\geq 2$ in $\mathcal G$. 
	
	Take  a point $Q=(p,P)\in \mathcal G$  so that  $P=(\ell, p_1+p_2,q_1+\cdots+q_{d-4})$ with $p\in\{p_1,p_2\}$. Choose a homogeneous coordinate system $t_0,\ldots,t_3$ on $\mathbb P^3$ and $g,h\in\mathbb C[t_0,t_1]$ so that $p=(1:0:0:0)$, $\ell=(t_2=t_3=0)$, and $P=(\ell, [g],[h])$. By renumbering we can set $p=p_1$.

     \medskip
    {\bf Case 1:} Assume that $p_1\neq p_2$ and $\{p_1,p_2\}\not\subset\{q_1,\ldots,q_{d-4}\}$. We only prove the case  $p_1\not\in \{q_1,\ldots,q_{d-4}\}$. The other case can be proved in a similar method. 
    So we  set    
	$g=t_0t_1$ so that $p=p_1=(1:0:0:0)$ and $p_2=(0:1:0:0)$. Since $p_1\not\in \{q_1,\ldots,q_{d-4}\}$, $t_1\nmid h$ and hence we can set   
	$$h=h_{d-4,0}t_0^{d-4}+\cdots+h_{1,d-5}t_0t_1^{d-5}+h_{0,d-4}t_1^{d-4}$$ for some  $h_{i,j}\in\mathbb C$ and $h_{d-4,0}=1$.  Let \begin{center}$r_1=t_0^2, r_2=t_1^2,  s_1=t_0^{d-5}t_1, s_2=t_0^{d-6}t_1^2 ,\ldots,s_{d-4}=t_1^{d-4}$ \end{center}so that $<g,r_1,r_2>=\mathbb C[t_0,t_1]_2$ and $<h,s_1,\ldots,s_{d-4}>=\mathbb C[t_0,t_1]_{d-4}$. 	
	
	For $Y\in \psi_1(\psi_2^{-1}(Q))$, let $f$, $x_{i_0,i_1,i_2}$ and $y_{i_0,i_1,i_2,i_3}$ be as in Notation~\ref{n.coord WP}.  Let us use $\{x_{i_0,i_1,i_2}, y_{i_0,i_1,i_2,i_3} \}\setminus \{x_{d-1,0,0}, y_{d-1,0,0,0}\}$ as  variables of the homogeneous coordinate ring of $\psi_2^{-1}(Q)\cong \mathbb P^{\dim W-d-2}$. 
	
	Choose an affine coordinate system $a_0,a_1,b_0,b_1,u_1,u_2,v_1,\ldots,v_{d-4}$ on an affine open neighborhood $U\cong\mathbb C^{d+2}$ of $P$ in $\mathcal F$ so that $(a_0,a_1,b_0,b_1,u_1,u_2,v_1,\ldots,v_{d-4})\in U$ represents $P'=(\ell', p_1'+p_2', q_1'+\cdots+q_{d-4}')\in \mathcal F$ such that 
	$\ell'=(t_2=a_0t_0+a_1t_1, t_3=b_0t_0+b_1t_1)$,  
	$$\ell'\cap (g+u_1r_1+u_2r_2=0)=p_1'+p_2'\mbox{ and }\ell'\cap (h+v_1s_1+\cdots+v_{d-4}r_{d-4}=0)=q_1'+\cdots+q_{d-4}'.$$ 
		Note that $(a_0,a_1,b_0,b_1,u_1,u_2,v_1,\ldots,v_{d-4})\in S(Y)\cap U$ if and only if 
	\begin{align*}
		&f(t_0,t_1,a_0t_0+a_1t_1, b_0t_0+b_1t_1)=\lambda(g+u_1r_1+u_2r_2)^2(h+v_1s_1+\cdots+v_{d-4}s_{d-4})
	\end{align*}
	for some $\lambda\in\mathbb C^*$. 	
	Using this and the implicity function theorem,  in some suitable analytic neighborhood of $P$, we can  identify $S(Y)$ with the analytic  subscheme of  $U\cong C^{d+2}_{a_0,a_1,b_0,b_1,u_1,u_2,v_1,\ldots,v_{d-4}}$  defined by zeros of  $d$ analytic equations with smallest terms as  follows: 	
	\begin{align*}
		u_1^2&=a_0^2x_{d-2,0,1}+a_0b_0y_{d-2,0,1,0}+b_0^2y_{d-2,0,0,1}\\
		2u_1&=a_0x_{d-2,1,0}+a_1x_{d-1,0,0}+b_0y_{d-2,1,0,0}+b_1y_{d-1,0,0,0}\\	
		2u_1h_{d-6,2}+2u_2h_{d-4,0}+v_1&=a_0x_{d-4,3,0}+a_1x_{d-3,2,0}+b_0 y_{d-5,1,0,0}+b_1y_{d-3,2,0,0}\\
		&\vdots\\
		2u_1h_{d-k-3,k-1}+2u_2h_{d-k-1,k-3}+v_{k-2}&=a_0x_{d-k-1,k,0}+a_1x_{d-k,k-1,0}+b_0 y_{d-k-1,k,0}+b_1y_{d-k,k-1,0,0}\\
		&\vdots\\
		2u_1h_{0,d-4}+2u_2h_{2,d-6}+v_{d-5}&=a_0x_{2,d-3,0}+a_1x_{3,d-4,0}+b_0y_{2,d-3,0,0}+b_1y_{3,d-4,0,0}	\\
		2u_2h_{1,d-5}+v_{d-4}&=a_0x_{1,d-2,0}+a_1x_{2,d-3,0}+b_0y_{1,d-2,0,0}+b_1y_{2,d-3,0,0}\\
		2u_2&=a_0x_{0,d-1,0}+a_1x_{1,d-2,0}+b_0y_{0,d-1,0,0}+b_1y_{1,d-2,0,0}\\
		0&=a_1x_{0,d-1,0}+b_1y_{0,d-1,0,0}.
	\end{align*} 
	
	Let $V\subset U$ be the linear subspace  given by the last  $d-1$ linear equations above. 
	We note that $V$ has dimension $3$ except only when $x_{0,d-1,0}=y_{0,d-1,0,0}=0.$ We assume that $x_{0,d-1,0}\neq0$ or $y_{0,d-1,0,0}\neq 0$ so that  $\dim V=3$.  
	Then $S(Y)$ is a surface with a double point at $P$ and the Zariski tangent space of $S(Y)$ at $P$ can be identified with  $V$. 
	
	Let  $\bf Q$ be the quadratic form on $U\cong\mathbb C^{d+2}_{a_0,a_1,b_0,b_1,u_1,u_2,v_1,\ldots,v_{d-4}}$ given by the following $(d+2)\times (d+2)$ matrix :
	$$\bf Q=\begin{pmatrix} x_{d-2,0,1}& 0& \frac{1}{2}y_{d-2,0,1,0}&0&0&0&\cdots\\
		0&0&0&0&0&0&\cdots\\
		\frac{1}{2}y_{d-2,0,1,0}&0&y_{d-2,0,0,1}&0&0&0&\cdots\\
		0&0&0&0&0&0&\cdots\\
		0&0&0&0&-1&0&\cdots\\
		\vdots&\vdots&\vdots&\vdots&\vdots&\vdots&\vdots\\
	\end{pmatrix}.$$
	Let $\bf Q^*$ be the quadratic form  on $V$ defined by the restriction of $\bf Q$. 
	   If $\bf Q^*$ has rank $\leq 1$, then there are no independent two vectors  $\bf v_1, \bf v_2\in V$ such that ${\bf v_1} {\bf Q} {\bf v_1}^t\neq 0$ and ${\bf v_2} {\bf Q} {\bf v_2}^t\neq 0$.  Take two vectors $${\bf a}=(1,0,0,0,M_1,M_2,L_1,\ldots,L_{d-4})\in V$$ and  $${\bf b}=(0,0,1,0,M_1',M_2',L_1',\ldots,L_{d-4}')\in V.$$   Here $M_i, M_i', L_i, L_i'$ are determined by $d-1$ linear equations above defining $V$.  Since ${\bf a} {\bf Q} {\bf a}^t=x_{d-2,0,1}-M_1^2$, ${\bf b}{\bf Q} {\bf b}^t=y_{d-2,0,0,1}-M_1'^2$ and  $({\bf a}+{\bf b}){\bf Q}({\bf a}+{\bf b})^t=x_{d-2,0,1}+y_{d-2,0,1,0}+y_{d-2,0,0,1}-(M_1+M_1')^2$, it follows that  the rank of $\bf Q^*$ is $\leq 1$ only when $$x_{d-2,0,1}-M_1^2=x_{d-2,0,1}+y_{d-2,0,1,0}+y_{d-2,0,0,1}-(M_1+M_1')^2=0,$$ $$y_{d-2,0,0,1}-M_1'^2=x_{d-2,0,1}+y_{d-2,0,1,0}+y_{d-2,0,0,1}-(M_1+M_1')^2=0,$$ or  $$x_{d-2,0,1}-M_1^2-M_1^2=y_{d-2,0,0,1}-M_1'^2=0.$$ Therefore for any $Y\in \psi_2^{-1}(Q)$ outside codimension 2 subscheme of $\psi_2^{-1}(Q)$, $S(Y)$ is surface  with a double point of rank $\geq 2$ at $P$. 
	
	\medskip 
	{\bf Case 2:} Assume that $p_1=p_2\not\in\{q_1,\ldots,q_{d-4}\}$. We may set that  $g=t_1^2$ and  $$h=h_{d-4,0}t_0^{d-4}+\cdots+h_{1,d-5}t_0t_1^{d-5}+h_{0,d-4}t_1^{d-4}$$ with $h_{d-4,0}=1$ so that $p=p_1=p_2=(1:0:0:0)$. 
		Set \begin{center}$r_1=t_0t_1, r_2=t_0^2,  s_1=t_0^{d-5}t_1, s_2=t_0^{d-6}t_1^2,\ldots,s_{d-5}=t_0t_1^{d-5},s_{d-4}=t_1^{d-4}$ \end{center} so that $<g,r_1,r_2>=\mathbb C[t_0,t_1]_2$ and $<h,s_1,\ldots,s_{d-4}>=\mathbb C[t_0,t_1]_{d-4}$. 
By the same argument as in case 1 we can show that in some analytic neighborhood of $P$, $S(Y)$ is  identified with the analytic subscheme of $U\cong\mathbb C^{d+2}_{a_0,a_1,b_0,b_1,u_1,u_2,v_1,\ldots,v_{d-4}}$ defined by $d$ analytic equations with smallest degree terms as follows: 
	
	\begin{align*}
	u_2^2=&a_0^2x_{d-2,0,1}+a_0b_0y_{d-2,0,1,0}+b_0^2y_{d-2,0,0,1}\\
	0=&a_0x_{d-2,1,0}+a_1x_{d-1,0,0}+b_0y_{d-2,1,0,0}+b_1y_{d-1,0,0,0}\\
	2u_2=&a_0x_{d-3,2,0}+a_1x_{d-2,1,0}+b_0y_{d-3,2,0,0}+b_1y_{d-2,1,0,0}\\
	2u_1+2u_2h_{d-5,1}=&a_0x_{d-4,3,0}+a_1x_{d-3,2,0}+b_0 y_{d-4,3,0,0}+b_1y_{d-3,2,0,0} \\
	2u_1h_{d-6,2}+2u_2h_{d-7,3}+v_1=&a_0x_{d-6,5,0}+a_1x_{d-5,4,0}+b_0 y_{d-6,5,0,0}+b_1y_{d-5,4,0,0} \\
	&\vdots\\
	2u_1h_{d-k-1,k-3}+2u_2h_{d-k-2,k-2}+v_{k-4}=&a_0x_{d-k-1,k,0}+a_1x_{d-k,k-1,0}+b_0 y_{d-k-1,k,0}+b_1y_{d-k,k-1,0,0}\\
	&\vdots\\
	2u_1h_{1,d-5}+2u_2h_{0.d-4}+v_{d-6}=&a_0x_{1,d-2,0}+a_1x_{2,d-3,0}+b_0y_{1,d-2,0,0}+b_1y_{2,d-3,0,0}\\
	2u_1h_{0,d-4}+v_{d-5}=&a_0x_{0,d-1,0}+a_1x_{1,d-2,0}+b_0y_{0,d-1,0,0}+b_1y_{1,d-2,0,0}\\
	v_{d-4}=&a_1x_{0,d-1,0}+b_1y_{0,d-1,0,0}
\end{align*}
	Then the subspace $V\subset U$ determined by the above last $d-1$ linear equations has dimension 3  except only when 	
$x_{d-2,1,0}=x_{d-1,0,0}=y_{d-2,1,0,0}=y_{d-1,0,0,0}=0 $. 
Assume that $V$ has dimension 3. 
By the same reason as before, for $Y\in \psi_2^{-1}(Q)$ outside codimension $\geq 2$  subscheme of $\psi_2^{-1}(Q)$,  $S(Y)$ is a surface with a double point of rank $\geq 2$ at $P$. 
\end{proof}

\section{Proof of Theorem~\ref{t.main}}
We now want to show our main theorem. In fact its proof immediately comes from the previous results. 
 \begin{proof}[Proof of Theorem~\ref{t.main}]For any $Y$ in the outside codimension 2 subscheme of $W$, $Y$ has only isolated singularities and $Y^*$ is an irreducible hypersurface of $(\mathbb P^3)^*$. Therefore if $L\subset W$ is a general line then $S(Y)$ is a projective surface for any $Y\in L$ by Proposition \ref{p.congruence}.
 
 (ii)  comes from  Theorem \ref{t.singular hypersurface}. For the proof of (i) we only need to show that the singular point in Theorem \ref{t.singular locus} is  isolated. Assume that $P=(\ell, p_1+p_2, q_1+\cdots+q_{d-4})$ is a singular point of $S(Y)$, $Y\in L$, such that $Y$ is smooth at $p_1$ and $p_2$. Proposition \ref{p.isolated singularity} says that there are at most finite such $P$. Therefore if $P$ is not an isolated singular point of $S(Y)$, then $Y$ is singular and $P$ should be in the closure of $\Gamma_Y$, which is impossible because  $\Gamma_Y$ is a projective curve (see Theorem \ref{t.singular hypersurface}). 

We now want to show (iii). Let $b\in L\setminus\{b_1,\ldots,b_N\}$. 
 Since $S(Y_b)$ is a projective surface, its class $[S(Y_b)]=c_d(R)$ where $R$ is the vector bundle of rank $d$ on $\mathcal F$ described in Section \ref{s.Fano congruence of bitangents}. Furthermore $S(Y_b)$ has only isolated singularities by (i). Thus by Serre's condition for normality it follows that $S(Y_b)$ is normal. 
 
 Recall that  $\phi_{Y_b}:S(Y_b)\rightarrow B(Y_b)$ is one to one and immersed over $B(Y_b)\setminus \{S_{2^3,1^{d-6}}(Y_b)\cup S_{3,2,1^{d-5}}(Y_b)\}$ (see Theorem~\ref{t.smooth}). Thus in order to prove that $\phi_{Y_b}: S(Y_b)\rightarrow B(Y_b)$ is the normalization, we only need to show that  $\dim S_{2^3,1^{d-6}}(Y_b)=\dim S_{2,3,1^{d-5}}(Y_b)=1$. In fact this implies that $\phi_{Y_b}$ is a finite birational morphism, and hence it  is the normalization because $S(Y_b)$ is normal. Let $\mathcal K\subset W\times\mathbb P(\Sym^3\mathcal I^{\vee})\times_G\mathbb P(\Sym^{d-6}\mathcal I^{\vee})$ be the space of $(Y,\ell,p_1+p_2+p_3, q_1+\cdots+q_{d-6})$ such that  $\ell\cap Y$ contains $2p_1+2p_2+2p_3+q_1+\cdots+q_{d-6}$ as a subscheme. Let $\nu_1:\mathcal K\rightarrow W$ and $\nu_2:\mathcal K\rightarrow G$ be the natural projections. Given $\ell\in G$, $\nu_2^{-1}(\ell)$ is isomorphic to  $\mathbb P^3\times \mathbb P^{d-6}\times\mathbb P^{\dim W-d} $ and hence $\mathcal K$ is smooth irreducible of dimension $\dim W+1$. Therefore general fiber of $\nu_1$ is a smooth irreducible projective curve. Since $\mathcal K$ is irreducible, the fiber $\nu_1^{-1}(Y)$ for all $Y$ outside codimension 2 subscheme of $W$ has dimension 1. It follows that if we take general line $L\subset W$, then $\nu_2(\nu_1^{-1}(Y))=S_{2^3,1^{d-6}}(Y)\subset B(Y)$ has dimension 1 for all $Y\in L$. 
Similarly we can show that $\dim S_{3,2,1^{d-5}}(Y)=1$ for all $Y\in L$.  We are done. 
\end{proof}

\begin{remark}\begin{itemize}
\item[(i)] Applying the same arguments in \cite{L84} it follows  that if $H_1(S(Y),\mathbb Q)\neq 0$ for general $Y$, then $S(Y_b)$ in Theorem~\ref{t.main}~(ii) is irreducible. We expect $S(Y_b)$ to be irreducible.
\item[(ii)] Assume that $d=6$ and $Y$ is smooth. Let $w:X\rightarrow \mathbb P^3$ be the double cover of $\mathbb P^3$ branched on $Y$. Then $X$ is a Fano threefold with $-K_X=2H$ where $H$ is the pullback of the hyperplane class on $\mathbb P^3$. By a conic on $X$ we mean a rational curve on $X$ with $(-K_X)$-degree 2. There are two families of conics in $X$.  One is the family of rational curves $C$ in $X$ such that its image $w(C)$ is a conic in $\mathbb P^3$ meeting $Y$ with even contact order at  each intersection point. This family had been studied in \cite{CV86}. They obtained a result  similar  to Theorem \ref{t.main}, and applied it in order to show that the associated Abel-Jacobi map is an isomorphism for general $Y$. The other family is the space of rational curves $C$ in $X$ such that $C=w^*(\ell)$ for some bitangent line $\ell$ of $Y$. It follows that in the condition of Proposition~\ref{p.congruence} the dimension of this family is equal to the expected dimension 2.
We remark that in contrast to the case in \cite{CV86}, the Abel-Jacobi map associated to the second family of conics on $X$ is trivial. 
\end{itemize}

\end{remark}

\end{document}